\documentclass{amsart}

\usepackage{comment}

\usepackage{dsfont}
\usepackage{amsmath}
\usepackage{amssymb}
\usepackage{graphicx}

\usepackage{amssymb}
\usepackage{comment}

\usepackage{amsfonts}

\usepackage{soul}

\usepackage{mathpazo}
\usepackage{color}
\usepackage{yfonts}

\usepackage{paralist}

\usepackage{stmaryrd}

\usepackage{amsxtra}

\DeclareMathOperator{\diff}{Diff}

\newcommand{\norm}[1]{\| #1\|}

\def\a{          \alpha}

\def\cB{          \mathcal B}

\def\clr{   \color{red}}

\def\clb{   \color{black}}

\usepackage[colorlinks,linkcolor=blue,citecolor=blue,urlcolor=blue]{hyperref}

\def\cE{             \mathcal E}

\def\vol{     {\rm Vol}}

\let\cal\mathcal

\def \R{{\mathbb R}}

\def \Z{{\mathbb Z}}
\def \N{{\mathbb N}}

\def \ta{{\tilde{\alpha}}}

\def \rank{{\rm rank}}

\newcommand{\T}{{\mathbb T}}
\newcommand{\prf}{{\begin{proof}}}
\newcommand{\epf}{{\end{proof}}}

\newcommand{\ary}{\begin{eqnarray}}
\newcommand{\eary}{\end{eqnarray}}

\newcommand{\aryst}{\begin{eqnarray*}}
\newcommand{\earyst}{\end{eqnarray*}}

\newcommand{\enmt}{\begin{enumerate}}
\newcommand{\eenmt}{\end{enumerate}}

\newtheorem{theo}{\sc Theorem}
\newtheorem{prop}{\sc Proposition}

\newtheorem{claim}{\sc claim}

\newtheorem{lemma}{\sc lemma}

\newtheorem{cor}{\sc corollary}

\newtheorem{conj}[theo]{\sc Conjecture}

\theoremstyle{definition}

\def\bee{\begin{equation}}
\def\eee{\end{equation}}

\newtheorem{defi}{\sc Definition}

\theoremstyle{rema}

\newcommand{\Diff}{\text{Diff}}

\newcommand{\pdvr}[2]
{\dfrac{\partial^{#2} #1}{\partial \theta^{#2_1} \partial r^{#2_2}}}
\newcommand{\pdvrs}[2]
{\partial^{#2} #1 /\partial \theta^{#2_1} \partial r^{#2_2}}
\newtheorem{thm}{\sc Theorem}

\numberwithin{equation}{section}

\author{Aaron Brown}
\thanks{Brown was supported by NSF No.1752675}
\address[Brown]{Northwestern University. Evanston, IL 60208, USA}
\email{awb@northwestern.edu}

\author{Danijela Damjanovi\'c}

\thanks{ Damjanovi\'c was supported by Swedish Research Council grant VR2015-04644}
\address[Damjanovi\'c]{Department of Mathematics, Kungliga Tekniska högskolan, Lindstedtsvägen 25, SE-100 44 Stockholm, Sweden} 
\email{ddam@kth.se}

\author{Zhiyuan Zhang}
\thanks{ Zhang was supported by the National Science Foundation under Grant No. DMS-1638352. Zhang also thanks the support and hospitality of IAS and KTH since part of this work is done during his postdoc there}
\address[Zhang]{
	Institut Galil\'ee
	Universit\'e Paris 13, CNRS UMR 7539,
	93430 - Villetaneuse, France
}
\email{zhiyuan.zhang@math.univ-paris13.fr}


\begin{document}

\title[$C^1$ lattice actions]{$C^1$ actions on manifolds by lattices in Lie groups}

\date{\today}
\maketitle

\begin{abstract} 
In this paper we study Zimmer's conjecture for $C^1$ actions of lattice subgroup of a higher-rank simple Lie group with finite center  on compact manifolds.  We show that when the rank of an uniform lattice is larger than the dimension of the manifold, then the action factors through a finite group.
For lattices in $SL(n, \R)$, the dimensional bound is sharp.   
\end{abstract}

\section{Introduction}

Zimmer's conjecture for actions of higher-rank lattice on compact manifolds says that if the group is large with respect to the dimension of the manifold, then any such action should factor through a finite group. This conjecture is motivated by a long history of research, including the local rigidity results of Selberg \cite{Sel} and Weil \cite{Weil} on linear representation theory, the global rigidity results of Mostow \cite{Mostow}, the superrigidity theorem of Margulis \cite{Mar}, and the  cocycle superrigidity theorem of Zimmer \cite{Zim}. Since its introduction, Zimmer's conjecture has attracted considerable interests. 

For $C^0$ actions on the circle, the above conjecture is confirmed by Lifschitz, Witte Morris \cite{LiWi, Wit} for many non-unifrorm lattices. 
For $C^1$ actions on the circle, Burger-Monod \cite{BM} and Ghys \cite{Ghys} showed similar results for many other cases, including all lattices in higher rank simple Lie groups. 
 For $C^1$ area preserving actions on closed orientable surface with genus at least $2$, Zimmer's conjecture is proved by 
  Polterovich \cite{Pol} for non-uniform lattices. His result is then generalised by Franks-Handel in \cite{FH} to any $C^1$ action which preserves a Borel measure.  For analytic actions,  Ghys \cite{Ghys} studied the case where the manifold is a circle; Farb-Shalen \cite{FS} studied this conjecture under additional assumptions on the group and the manifold. For a very detailed survey on other  earlier results on Zimmer's program, we refer the readers to \cite{Fis2}.
  
  In recent breakthrough \cite{BFH1, BFH2}, Brown-Fisher-Hurtado proved the $C^2$ version\footnote{Their result can be improved with a bit more work to include $C^{1+\epsilon}$-actions.} of Zimmer's conjecture for all co-compact lattices \footnote{These results are generalized recently in \cite{BFH3} to all non-uniform lattices.} in real split simple Lie group and ${\rm SL}(n,\Z)$ using some previous progress made by Brown-Rodriguez Hertz-Wang in \cite{BRHW1,BRHW2} and Lafforgue, de Laat and de la Salle in \cite{La, dLdl, dlS}.
We refer the reader to Fisher's paper \cite{Fis1} for an excellent survey of the history and recent progress on Zimmer's conjecture.
The purpose of the present paper is to extend the results in \cite{BFH1, BFH2, BFH3} to $C^1$ actions, when the rank of the acting group is sufficiently large.

Compared to the previous results, there are 2 new ideas here. First is that while many results in Non-uniform Hyperbolic Theory fail or remain unknown in the $C^1$ setting, some of them continue to hold under the presence of suitable continuous splitting. In our case, we can apply a variant of Avila-Viana's invariance principle to an element in the kernel of all Lyapunov functionals to obtain the extra invariance needed to conclude the proof. For $C^2$ action, the idea to use action by an element in the  kernel of all fiberwise exponents was originally due to Sebastian Hurtado and appears in the Bourbaki notes of Cantat \cite{Can}. 
The second one is that we use the information extracted  by using strong property (T) to control the $L^p$ norms of the derivatives for sufficiently large $p$. This allows us to show that $C^1$ action is uniformly bounded under certain H\"older norm.  Then we use the resolution of the Hilbert-Smith conjecture for sufficiently H\"older actions to conclude the proof.

\section{Statement of the main results}

We first recall the statement of Zimmer's conjecture.

For a real semisimple Lie group $G$ with Lie algebra $\mathfrak{g}$, let
\begin{itemize}
	\item $v(G)$ denote the minimal codimension of proper parabolic subalgebras of $\mathfrak{g}$;
	\item $d(G)$ denote the minimal codimension of proper subalgebras of the compact real form of $\mathfrak{g}_\mathbb{C}$;
	\item $n(G)$ denote the minimal dimension of nontrivial real representations of $\mathfrak{g}$.
\end{itemize}
It is proved in \cite{St} that $v(G)<n(G)$.\footnote{We thank Jinpeng An for this remark.} 
\begin{conj}\label{C:Zimmer}
	Let $G$ be a connected real semisimple Lie group with finite center and without almost-simple factors of real rank less than $2$. Let $\Gamma < G$ be a lattice, $M$ be a compact manifold, $\alpha:\Gamma\to\Diff(M)$ be an action.
	\begin{itemize}
		\item[(1)] If $\dim(M)<v(G)$, then $\alpha$ preserves a Riemannian metric.
		\item[(2)] If $\dim(M)<\min\{v(G),d(G)\}$, then $\alpha(\Gamma)$ is finite.
		\item[(3)] If $\dim(M)<n(G)$ and $\alpha$ preserves a volume density, then $\alpha$ preserves a Riemannian metric.
		\item[(4)] If $\dim(M)<\min\{n(G),d(G)\}$ and $\alpha$ preserves a volume density, then $\alpha(\Gamma)$ is finite.
	\end{itemize}
\end{conj}

The main result of this paper is the following generalisation of the results in \cite{BFH1, BFH3} to $C^1$ regularity.


\begin{thm}\label{thm small action general}
	Let $M$ be a compact manifold.
	Let $G$ be an almost simple real Lie group with finite center and with real-rank at least $2$, and let $\Gamma < G$ be a lattice.
	Let   $\alpha:  \Gamma \to  \diff^1(M)$ be a  group homomorphism.
	Assume that $\Gamma$ is an uniform lattice or $\Gamma = SL(n, \Z)$, and assume either that $\dim M < \min(\rank_\R(G), d(G))$, or that $\dim M \leq \min(\rank_\R(G), d(G)-1)$ and $\alpha(\Gamma) \subset \diff^1(M, {\rm vol})$.
	Then $\alpha$ has finite image.
\end{thm}

Compared to the main result in \cite{BFH1, BFH3} for almost-simple real Lie groups, in Theorem \ref{thm small action general} we have posed a different requirement on the dimension of the manifold. Indeed, we can deduce from \cite[Theorem 2.7]{BFH1} that for a group homomorphism $\a : \Gamma \to \diff^2(M)$, the conclusion of Theorem \ref{thm small action general} is true if $\rank_\R G$ is replaced by the \textit{minimal resonant codimension} $r(G)$ (see \cite[Definition 2.1]{BFH1}). 
We remark that under the conditions of  Theorem \ref{thm small action general}, we always  have that 
\aryst
r(G) \geq \rank_\R G.
\earyst

\begin{cor}\label{thm finite action}
Let $M$ be a compact manifold.  Let $\Gamma < G$ be a lattice. Let  $\alpha: \Gamma \to \diff^1(M)$ (resp. $\diff^1(M, {\rm vol})$) be a  group homomorphism.
Assume that $\Gamma$ is an uniform lattice or $\Gamma = SL(n, \Z)$, and assume 
that one of the following is true:
\enmt
\item $G = SL(n, \R)$, $\dim M < n-1$  (resp. $\leq n-1$) and $n \geq 3$;
\item $G = Sp(2n, \R)$, $\dim M < n$ (resp. $\leq n$) and $n \geq 2$;
\item $G = SO(n, n)$, $\dim M < n$  (resp. $\leq n$) and $n \geq 4$;
\item $G = SO(n, n+1)$, $\dim M < n-1$ (resp. $\leq n-1$) and $n \geq 3$.
\eenmt
Then $\alpha$ has  finite image.  
\end{cor} 

When $\a$ is a $C^2$ action, the conclusion of  Theorem \ref{thm small action general} is already obtained in \cite{BFH1, BFH3}.  
Moreover, when $G = Sp(2n, \R)$, $SO(n, n)$ or $SO(n,n+1)$, the dimension bound in Corollary \ref{thm finite action} is not optimal.
However, when $G = SL(n, \R)$, we have
\aryst
r(G) = \rank_\R G = n-1.
\earyst
By considering the  actions of $SL(n, \R)$ by projective transformations on $\mathbb{P}(\R^n)$, and by the affine transformations on $\T^n$, we see that  Corollary \ref{thm finite action} has optimal bounds for $G = SL(n, \R)$. We note the for $C^0$ action by $SL(n,\Z), (n \geq 3)$  on compact manifold with $\chi(M) \neq 0 \mod 3$, the finite image property of $\a$ is proved by Ye in \cite{Ye}.

%
%


%
%

The proofs of the results in this paper follow closely the strategy in \cite{BFH1}. We recommend the reader to have this paper close at hand as we make many references to these works, although we also repeat some of the main arguments for reader's convenience. Below we first describe the general strategy of the proofs in  \cite{BFH1,BFH2, BFH3}, and then we point out the main new ideas and modifications we make here in order to obtain results in $C^1$ regularity.

\section{Review of the work of Brown, Fisher and Hurtado, and outline of the proof}

\noindent{\bf Step 1: Uniform subexponential growth.}

We fix a finite set of symmetric generators  for $\Gamma$, denoted by $S =\{\gamma_i\}$. 
For any $\gamma \in \Gamma$, we let $\ell(\gamma)$ denote the word-length distance from $\gamma$ to the identity relative to $S$. In other words, $\ell(\gamma)$ is the smallest integer $k$ such that $\gamma$ may be represented by a product $\zeta_1 \cdots \zeta_k$ where $\zeta_j \in S$ for each $1 \leq j \leq k$.

We first recall the following notion.
\begin{defi}
Let $\a : \Gamma \to \diff^{1}(M)$ be an action of $\Gamma$ on a compact manifold $M$ by $C^1$ diffeomorphisms. We fix a background $C^\infty$ Riemannian metric on $M$. We say that $\a$ has \textit{uniform subexponential growth of derivatives} if for every $\varepsilon > 0$ there is a constant  $C_\varepsilon > 0$ such that for all $\gamma \in \Gamma$ we have
\aryst
\norm{D\a(\gamma)} \leq C_{\varepsilon} e^{\varepsilon \ell( \gamma)}.
\earyst
It is clear that the above definition is independent of the choice of the metric on $M$ or the generating set $S$.
\end{defi}

The main result of Step 1 is the following.

\begin{prop}\label{thm small derivates}
Let $M$ be a  compact manifold, and let $G$ be a connected, almost-simple real Lie group with finite center and whose real-rank is at least $2$. Let $\Gamma < G$ be a lattice. 
	Let   $\alpha:  \Gamma \to  \diff^1(M)$ be a  group homomorphism.
Assume $\Gamma$ is an uniform lattice or $\Gamma = SL(n, \Z)$, and assume either that $\dim M <	\rank_\R(G)$, or that $\dim M \leq \rank_\R(G)$ and $\alpha(\Gamma) \subset \diff^1(M, {\rm vol})$.
Then $\alpha$ has  uniform subexponential growth of derivatives.
\end{prop}

We prove Proposition \ref{thm small derivates} following the same scheme in \cite{BFH1}.
As in \cite{BFH1}, we define the suspension space $M^\a$ as the quotient of $G \times M$ by  $\Gamma$-action $(g,x) \mapsto (g\gamma, \a(\gamma^{-1})x)$.  
We recall that $M^\a$ is a fiber bundle over $G / \Gamma$ with fibers modeled on $M$. Moreover $M^\a$ is equipped with a left $G$-action, denoted by $\tilde \a$, by diffeomorphisms which preserves the foliation into fibers.
We present the construction of $M^\a$ and its further properties in Section \ref{sec: Suspension constuction and its properties}.

As the $G$-action preserves the foliation into fibers of $M^\a$, we may consider the restriction of $D\ta$ to the subbundle $E^F :={\rm Ker}(D\pi)$ tangent to the fibers of $M^\a$. 
Let $A$ be the maximal split torus of $G$, and let $\mu$ be an $A$-ergodic $A$-invariant measure on $M^\a$.
We can associate to $\mu$ and the derivative $A$-cocycle $D\ta|_{E^F}$ a set of \textit{fiberwise Lyapunov functionals} $\lambda^F_i : Lie(A) \to \R$, $1 \leq i \leq k$ by the higher-rank Oseledec's theorem (see, e.g., \cite[Part I, Theorem 2.4]{BRHW1}). We refer the reader to \cite[Proposition 3.3]{BFH1} for the definition and properties of Lyapunov functionals.
The \textit{maxmal fiberwise Lyapunov exponent} for $a \in A$ with respect to an $a$-invariant probability measure $\mu$ is defined as
\aryst
\lambda^F_+(a,\mu) = \inf_{n \to \infty}\frac{1}{n} \int \log \norm{D\ta(a^n) |_{E^F(x)}} d\mu(x).
\earyst

We have the following.
\begin{prop}\label{prop. prop3.7inbfn1}
	Suppose that $\Gamma$ is an uniform lattice or $\Gamma = SL(n, \Z)$, and $\a$ fails to have uniform subexponential growth of derivatives. 
There exists an $s \in A$ and an $A$-invariant Borel probability measure $\mu$ on $M^\a$ with $\lambda^F_{+}(s, \mu) > 0$
such that $\pi_*\mu$ is the Haar measure on $G / \Gamma$.
\end{prop}
When $\Gamma$ is an uniform lattice, the above proposition is just \cite[Proposition 3.7]{BFH1}.
When $\Gamma= SL(n, \Z)$, the above proposition is proved in \cite{BFH2} even though it is not explicitly stated as a single proposition. Indeed, we can define the measure $\mu$ in Proposition \ref{prop. prop3.7inbfn1} by \cite[Proposition 5.10]{BFH2} as a limit of a sequence $\mu_n$; and by \cite[Proposition 5.6]{BFH2} and the two paragraphs below it, we see that $\pi_* \mu$ is the Haar measure on $G/\Gamma$.

This is the only place where we have used the hypothesis that $\Gamma$ is an uniform lattice or $\Gamma = SL(n, \Z)$.
In a recent paper of Brown-Fisher-Hurtado \cite[Prop 8.1]{BFH3}, they have generalised 
Proposition \ref{prop. prop3.7inbfn1} to any lattice in $G$. Admitting their results, all of the results in the present paper hold for arbitrary lattices.

To complete the proof of Proposition \ref{thm small derivates}, it remains to show the following.
\begin{prop}\label{prop. atoginvariance}
Let $\mu$ be an $A$-invariant Borel probability measure on $M^\a$ such that $\pi_*\mu$ is the Haar measure on $G / \Gamma$.  If either that $\rank_\R G  > \dim M$, or that $\rank_\R G \geq \dim M$ and  $\a(\Gamma) \subset \diff^1(M, {\rm vol})$, then $\mu$ is $G$-invariant.
\end{prop}

Let $a\in G$ be a $\R$-semisimple element. 
  The unstable, resp. stable, subgroup for $a$ are respectively 
\aryst
H^u &:=& \{g \mid \lim_{n \to - \infty}a^{n}ga^{-n} = e \}, \\
H^{s} &:=&  \{g \mid \lim_{n \to + \infty}a^{n}ga^{-n} = e \}.
\earyst

\begin{prop} \label{prop. inv}
Let $a \in A$ be an $\R$-semisimple element. Suppose $\mu$ is an $a$-invariant $a$- ergodic probability measure on $M^\a$ such that
\enmt
\item $\pi_*\mu$ is the Haar measure on $G /\Gamma$, and
\item all fiberwise Lyapunovv exponents of $D\tilde\alpha(a)$ are non-positive.
\eenmt
Then $\mu$ is $H^u$-invariant.
\end{prop}

The proof of Proposition \ref{prop. inv} will be given in Section \ref{sec. proof of prop inv}.
We are ready to deduce Proposition \ref{prop. atoginvariance} from Proposition \ref{prop. inv}.
\begin{proof}[Proof of Proposition \ref{prop. atoginvariance}]
 We can assume without loss of generality that $\mu$ is $A$-ergodic, otherwise we may replace $\mu$ by any one of its $A$-ergodic components. This is because any $A$-ergodic component of $\mu$ projects to some $A$-ergodic component of $\pi_*\mu$; while by hypothesis $\pi_*\mu$ is the Haar measure on $G/\Gamma$ which is itself $A$-ergodic by Moore's ergodicity theorem (see for instance \cite{Moo} or \cite[Theorem 2.2.6]{Zim2}).
 This allows us to define fiberwise Lyapunov functionals. 
We denote by $\lambda^F_1, \cdots, \lambda^F_k$ the total collection of distinct fiberwise Lyapunov functionals.
We have that $k \leq \dim M$. Moreover, notice that when $\alpha(\Gamma) \subset \diff^1(M, {\rm vol})$, the sum of  all Lyapunov functionals (considered with multiplicities) is zero. 
Then under the condition of the proposition,  we can pick an arbitrary element $a \in (\cap_{i=1}^k \exp( {\rm Ker}(\lambda^F_i))) \setminus \{ e \}$ such that
\aryst
\lambda^F_{+}(a, \mu) =  \lambda^F_{+}(a^{-1}, \mu) =  0.
\earyst
Then all $a$-ergodic components of $\mu$ have vanishing fiberwise Lyapunov exponents. 
By Proposition \ref{prop. inv}, we deduce that $\mu$ is $H^{u}$-invariant. By symmetry, we also have that $\mu$ is $H^s$-invariant. As $G$ is almost-simple, $G$ is generated by $H^u$ and $H^s$. Consequently, $\mu$ is $G$-invariant.
\end{proof}

\begin{proof}[Proof of Proposition \ref{thm small derivates}]
Assume that $\a$ fails to have uniform subexponential growth of derivatives.
Then by Proposition \ref{prop. prop3.7inbfn1},   there is a $s \in A$ and an $A$-invariant measure $\mu$ such that
$\lambda^F_{+}(s, \mu) > 0$ and  $\pi_*\mu$ is the  Haar measure on $G / \Gamma$.  
By Proposition \ref{prop. atoginvariance}, we deduce that $\mu$ is $G$-invariant.
Recall that $n(G) > \rank_\R G$ where  $n(G)$ denotes the minimal dimension of a non-trivial real representation of the Lie algebra of $G$.
By Zimmer's cocycle superrigidity theorem (we use the version by Fisher-Margulis in \cite[Theorem 1.4]{FM}.  We refer the readers to \cite{Zim, Zim2, ZimmerWitteMorris} for some earlier results), the $G$-action preserves a measurable metric on $E^F$. This contradicts that $\lambda^F_+(s, \mu) > 0$.
Thus $\a$ must has uniform subexponential growth of  derivatives.
\end{proof}

\smallskip

\noindent{\bf Step 2: Strong property $(T)$ and averaging.}

In this step, we follow \cite{BFH1} to construct a $\Gamma$-invariant continuous distance by using the strong property $(T)$ of $\Gamma$ proved by Lafforgue, de Laat and de la Salle in \cite{La, dLdl, dlS}.
The main result of this step is the following proposition whose proof will be given in Section \ref{sec: Subexponential growth and invariant metric}. 
\begin{prop}\label{prop subexp and invariant metric}
If   $\alpha$ has uniform subexponential growth of derivatives, then there exists a distance $\overline{d}: M \times M \to [0, \infty)$ that is invariant by the $\Gamma$-action $\alpha$. Moreover, for any $\beta \in (0,1)$, the set $\alpha(\Gamma)$ is precompact in  ${\rm Hol}$-${\rm Homeo}^{\beta}(M)$, the space of $\beta$-bi-H\"older homeomorphisms of $M$ with respect to the background Riemannian distance.
\end{prop}

Proposition \ref{prop subexp and invariant metric} replaces \cite[Theorem 2.9]{BFH1}. In \cite{BFH1}, the authors study a $C^2$-action of $\Gamma$, and the induced $\Gamma$ action on $W^{1,p}(S^2(T^*M))$, the Sobolev space of all the sections $\varphi$ of the bundle of symmetric two forms  $S^2(T^*M)$ such that both $\varphi$ and $D\varphi$ are $L^p$ with respect to the Lebesgue measure. Then the strong property $(T)$ and the unifrom subexponential growth of derivatives give us the $\Gamma$-invariant section in  $W^{1,p}(S^2(T^*M))$ which is continuous if $p$ is sufficiently large. The above method can be adapted to the case where the action is $C^{1 + \epsilon}$ for any $\epsilon > 0$.

In our case, $\a$ is only $C^1$, and consequently $\a$ does not induce a $\Gamma$ action on $W^{1,p}(S^2(T^*M))$. We consider instead the induced $\Gamma$-action on  $L^p(S^2(T^*M))$, and obtain a $L^p$ $\a$-invariant section of $S^2(T^*M)$. We use the exponential convergence inherited from the strong  property $(T)$ and Cauchy inequality to bound  the Sobolev norms of the $\Gamma$-action. 

To make use of Proposition \ref{prop subexp and invariant metric}, we also need the solution of Hilbert-Smith conjecture for sufficiently H\"older actions proved in \cite{RS, Mal}. We recall the statement here.
\begin{lemma}\label{lemmahilbertsmithholde}
	For any  $\beta \in (\frac{\dim M}{\dim M +1}, 1)$ the following is true: let $H$ be a compact topological group which admits a faithful action on $M$ by $\beta$-H\"older homeomorphisms. Then $H$ is a Lie group.
\end{lemma}

\begin{cor}\label{thm small action}
	Let $G, \Gamma	, \mu, \a$ be as in Theorem \ref{thm small action general}.
	Assume either that $\dim M <	\rank_\R(G)$, or that $\dim M \leq \rank_\R(G)$ and $\alpha(\Gamma) \subset \diff^1(M, {\rm vol})$.
	Then $\alpha$ factors through a compact Lie group. That is, there exist: a compact Lie group $H$; an injective group homomorphism $\iota: H \to {\rm Homeo}(M)$; and a  group homomorphism $\phi :  \Gamma \to H$ such that $\a =\iota \circ \phi$. 
\end{cor} 

\begin{proof}
	By Proposition \ref{thm small derivates}, the action  $\alpha$ has uniform subexponential growth of derivatives.
	We fix any $\beta \in (\frac{\dim M}{\dim M+1}, 1)$. By Proposition \ref{prop subexp and invariant metric}, the closure of $\a(\Gamma)$ in   ${\rm Hol}$-${\rm Homeo}^{\beta}(M)$, denoted by $K_0$,  is a  compact topological subgroup of ${\rm Homeo}(M)$.
	By Lemma \ref{lemmahilbertsmithholde}, we see that $K_0$ is a compact Lie group. 
\end{proof}

\smallskip

\noindent{\bf Step 3: Margulis superrigidity with compact codomain.}

After Step 1 and 2, we can apply precisely the same method as in \cite{BFH1} to show the finite image property. We refer the reader to \cite[Section 7]{BFH1} for details.

\begin{proof}[Proof of Theorem \ref{thm small action general}]
The proof is essentially contained in \cite[Section 7]{BFH1}. We reproduce it below for the convenience of the readers.

Let $H$ be the compact Lie group given by Corollary \ref{thm small action}, and let $\iota: H \to {\rm Homeo}(M)$ and $\phi :  \Gamma \to H$ be the associated  group homomorphisms.
Assume that $\alpha = \iota \circ \phi$ has infinite image.
Then by Margulis' arithmeticity theorem and superrigidity theorem, 
each almost simple factor of $H$ is a compact form of $G$. 
Since  $\iota: H \to {\rm Homeo}(M)$ is injective, there is some $x\in M$ such that $\iota(H) x$ contains a compacta homeomorphic to $K/C$ where $K$ is a compact form of $G$ and $C$ is a closed proper subgroup of $K$. This is impossible since by hypothesis $\dim (K/C) \geq d(G) > \dim M$.
\end{proof}

\section{Proof of Proposition \ref{prop. inv}}\label{sec. proof of prop inv}

\subsection{Suspension space} \label{sec: Suspension constuction and its properties}

In this subsection, we recall the suspension construction and the induced $G$-action in \cite[Section 2]{BRHW2}.

Let $\a$ be a $\Gamma$-action on $M$ by $C^1$ diffeomorphisms, i.e., $\a(gh) = \a(g)\a(h)$.
We consider the right acton by $\Gamma$ on $G \times M$ defined as
\aryst
(g,x)\cdot \gamma = (g\gamma, \alpha(\gamma^{-1})(x)), \quad \forall \gamma \in \Gamma
\earyst
and the left $G$-action
\aryst
a \cdot (g,x) =(ag, x), \quad \forall a \in G.
\earyst
Define the quotient manifold $M^{\a} := (G \times M)/ \Gamma$. Since the left $G$-action commutes with the right $\Gamma$-action, the left $G$-action descends to a left $G$-action on $M^{\a}$, denoted by $\ta$.
Since $\a$ is a $C^1$ action, $M^{\a}$ is naturally equipped with a $C^1$ manifold structure. The action $\ta$ is given by $C^1$ diffeomorphisms of $M^{\a}$. Moreover, denote by $\pi : M^{\a} \to G/\Gamma$ the projection induced by $G \times M \to G$, then $M^{\a}$ is a $C^1$ fiber bundle over $G/\Gamma$ induced by $\pi$ with fibers diffeomorphic to $M$. 

With a slight abuse of notation, we use $d(\cdot, \cdot)$ to denote both the right-invariant metric on $G$, and  the quotient metric  on $G/\Gamma$. 
We denote by $\nu$ the normalised left Haar measure on $G/\Gamma$.

By the construction in \cite[Section 2.2]{BFH2} (see also \cite[Section 2.1]{BRHW2} for the details),
there exists  a $C^1$ Riemannian metric $\langle \cdot, \cdot \rangle$ on $G \times M$ with the following properties:
\enmt
\item $\langle \cdot, \cdot \rangle$ is invariant under the right $\Gamma$-action,
\item for each $(g,x) \in G \times M$, under the canonical identification of the $G$-orbit of $(g,x)$ with $G$, the restriction of $\langle \cdot, \cdot \rangle$ to the $G$-orbit of $(g,x)$ coincides with $d_{G}$,
\item There exist a Siegel fundamental set $D \subset G$ for the right $\Gamma$-action (see \cite[VIII.1]{Mar} for the definition) containing the identity $e \in G$, and a constant $C_1 > 1$ such that for any $g_1,g_2 \in D$, the map $(g_1,x) \mapsto (g_2,x)$ distorts the restrictions of $\langle \cdot, \cdot \rangle$ to $\{g_1\} \times M$ and $\{g_2\} \times M$ by at most $C_1$.
\eenmt
We use $\langle \cdot, \cdot \rangle_g$ to denote the restriction of $\langle \cdot, \cdot \rangle$ to $\{g\} \times M$, and view it as a metric on $M$.
By item (1) above, we can equip $M^{\a}$ with the quotient metric of  $\langle \cdot, \cdot \rangle$.

We fix $\{\gamma_i\}$, a finite symmetric generating set for $\Gamma$. Let $\ell$ denote the word-length distance on $\Gamma$ relative to  $\{\gamma_i\}$. 
Given a fundamental domain $F_D \subset D$ for the right $\Gamma$-action on $G$, i.e., $G = F_D\Gamma$ and $F_D\gamma \cap F_D = \emptyset$ for $\forall \gamma \in \Gamma \setminus \{e\}$, the return cocycle $\beta: G \times G/\Gamma \to \Gamma$ associated to $F_D$ is defined as follows. For any $g \in G$,  $x \in G/\Gamma$, we set $\beta(g,x)$ to be the unique element $\gamma \in \Gamma$ such that $g\tilde{x} \in F_D \gamma$, where $\tilde{x}$ is the lift of $x$ in $F_D$. The following are from \cite{BFH2} whose proofs rely on \cite{LMR}.

\begin{lemma}\label{lem fd and precompactness}
If $F_D \subset D$ is a fundamental domain for the right $\Gamma$-action on $G$ such that $e \in F_D$, then 
there is a constant $C > 0$ such that for any $g \in G$, any $x \in G/\Gamma$, $$\ell( \beta(g,x)) < C d(g,e) + Cd(x, \Gamma) + C.$$
\end{lemma}

\begin{lemma}\label{prop integrability}
There is a constant $C > 0$ such that the following is true.
For  any $g \in G$, any $x \in G / \Gamma$, any $p \in \pi^{-1}(x)$ we have
\aryst
\log \norm{D_p\ta(g)} < C d(g,e) + Cd(x, \Gamma) + C.
\earyst
\end{lemma}

\subsection{Smooth cocycle}
Let $a$ be as in Proposition \ref{prop. inv}.
In various statements about typical points in $G /\Gamma$ in this rest of this section, we will always refer to the Haar measure $\nu$.

We first recall several basic definitions from measure theory following \cite[Appendix 1]{CFS}.
 A partition of a Lebesgue space $(Y, \cal{Y}, \mu_Y)$ (for the definition, see \cite[Appendix 1, Definition 4]{CFS}) is a family $\xi = \{ C \}$ of nonempty disjoint measurable subsets $C$ such that $\cup_{C \in \xi} C = M$. A subset $A \in {\cal Y}$ is said measurable with respect to $\xi$ if $A$ is a union of elements of $\xi$.  The partition $\xi$ is said to be measurable if there exists a countable collection of sets $\{ B_i \mid i \in I  \}$ which are measurable with respect to $\xi$ such that for any $C_1, C_2 \in \xi$ we can find an $i \in I$ such that either $C_1 \subset B_i, C_2 \not\subset B_i$ or $C_2 \subset B_i$, $C_1 \not\subset B_i$.  To any measurable partition $\xi$ we can assign a complete $\sigma$-algebra $\cal{B}_\xi \subset \cal{Y}$  consisting of the sets $A \in \cal{Y}$ which coincide modulo $\mu_Y$-null sets with one of the sets which is measurable with respect to $\xi$. In fact such correspondance is bijective (see \cite[Appendix 1, Section 3]{CFS}).

Following \cite{LedStr} and \cite[Sect 9.3]{MT}, we may find a measurable partition $\xi$ of $G / \Gamma$ with the following properties:
\enmt
\item[$(1)$] $\xi$ is subordinate to the partition of $G / \Gamma$ into $H^u$ orbits: for a.e. $x \in G / \Gamma$,
\enmt
\item[$(a)$] the atom $\xi(x)$ is contained in the orbit $H^u\cdot x$,
\item[$(b)$] the atom $\xi(x)$ is precompact in the orbit $H^u\cdot x$,
\item[$(c)$] the atom $\xi(x)$ contains a neighborhood of $x$ in the orbit $H^u\cdot x$,
\eenmt
\item[$(2)$] $\xi$ is $a$-decreasing, i.e., $a(\xi) \leq \xi$.
\eenmt
We also require that $\xi$ satisfies the following additional property:
\enmt
\item[$(3)$] There is a compact set $W \subset H^u$ such that for a.e. $x$
\aryst
\xi(x) \subset W \cdot x.
\earyst 
\eenmt

To build a partition $\xi$ satisfying $(1)$--$(3)$, we first let $\xi_0$ be a partition satisfying $(1)$ and $(2)$. Select a $\xi_0$-measurable subset $S \subset G / \Gamma$ with  positive $\nu$-measure such that the diameter of $\xi_0(x)$ is uniformly bounded in the $H^u \cdot x$-orbit for all $x \in S$. It is well-known that $a$ is ergodic with respect to the Haar measure $\nu$. Thus for a.e. $x \in G / \Gamma$, the following number is well-defined:
\aryst
n_x = \inf \{  n \in \N  \mid  a^n \cdot x \in S \}.
\earyst
We set
\aryst
\xi(x) = a^{-n_x} \xi_0(a^{n_x}\cdot x).
\earyst
Then $\xi$ still satisfies $(1)$ and $(2)$. Since ${\rm Ad}(a^{-1})$ is a contraction restricted to the Lie algebra of $H^u$, $\xi$ also satisfies $(3)$.

Since $\xi$ is measurable, we may apply \cite[Lemma 4.6]{AV} to find a measurable selection: there is a measurable map $\psi: G/\Gamma \to G/\Gamma$
such that $\psi$ is constant on every atom of $\xi$, and $\psi(x) \in \xi(x)$ for $\nu$-a.e. $x$.
Recall our choice of a Siegel fundamental set $D \subset G$ and fix a fundamental domain $F_D \subset D$ such that $e \in F_D$. Let $\bar{\psi} : G / \Gamma \to G$ be the map that assigns $x \in G / \Gamma$  the unique $g \in F_D$ with $\psi(x) = g\Gamma$. Note that $\bar \psi$ is $\xi$-measurable.

Since $H^u$ is horospherical for $a$, for a.e. $x \in G / \Gamma$ the map $H^u \to G /\Gamma$, $h \mapsto h \cdot x$ is injective.  Indeed, for a $\mu$-typical $x \in G/\Gamma$, there is a sequence $\{t_m\}_{m \geq 0}$ of positive numbers that tends to infinity such that $\{ a^{-t_m} \cdot x \}_{m \geq 0}$ is precompact. Then $h \mapsto h \cdot x$ must be injective on $H^u$ since each $H^u$-orbit is contracted by the backward iterates of $a$, and $G \to G/\Gamma$ is a local homeomorphism.
 For any such $x$,  we let $W_x$ be the inverse image of $\xi(x)$ under the map $H^u \to G/\Gamma$, $h \mapsto h  \cdot  \psi(x)$; and let $ \xi_1(x) = W_x \bar\psi(x)$. Notice that by definition $\pi(\xi_1(x)) = \xi(x)$, and $\xi_1(x) \cap F_D \neq \emptyset$.


As $F_D$ is a fundamental domain contained in $D$, we can choose a Borel trivialization associated to $F_D$, denoted by
\aryst
\iota: M^\a &\to& F_D \times M
\earyst
where for each $x \in G / \Gamma$, we identify $\iota|_{\pi^{-1}(x)}$ with a diffeomorphism  $\iota_x: \pi^{-1}(x) \to M$. Moreover, by the construction of the metric $\langle \cdot, \cdot \rangle$ on $D \times M$, we may assume that $\|\iota_x\|_{C^1}$ is uniformly bounded over all $x \in G/\Gamma$. 

Given a typical $x \in G / \Gamma$, let $u_x \in H^u$ be such that $x = u_x \cdot \psi(x)$. Set $g_x : \pi^{-1}(x) \to \pi^{-1}(\psi(x))$ to be
\aryst
g_x(y) = \ta(u_x^{-1})( y).
\earyst

Given $x \in G / \Gamma$, set $F_x : M \to M$ to be
\ary   \label{def F_x}
F_x(y) = \iota_{\psi(a^{-1} \cdot x)}( g_{a^{-1} \cdot \psi(x)}(\ta(a^{-1})( \iota_{\psi(x)}^{-1}(y)))).
\eary

Let $F: G / \Gamma \times M \to G/ \Gamma \times M$ be the measurable map
\ary \label{def F}
F(x,y) = (a^{-1}\cdot x, F_x(y) ).
\eary
Using $\{g_x\}$, we define a measurable map $\Phi : M^\a \to G /\Gamma \times M$  as follows:
\ary \label{def Phi}
\Phi(y)  =  (\pi(y), \iota_{\psi(\pi(y))} g_{\pi(y)}(y)).
\eary
Let $\mu$ be the $a$-ergodic $a$-invariant measure in Proposition \ref{prop. inv}, let $\mu^* = \Phi_* \mu$. 
\begin{claim}\label{claim. boreliso}
 $\Phi$ is a Borel isomorphism. Moreover,
 for $\mu$-a.e. $x \in G/\Gamma$, $\Phi$ is a $C^1$ diffemorphism from $\pi^{-1}(x)$ to $M$,
 and we have
\aryst
F \cdot \Phi  = \Phi \cdot \ta(a^{-1}).
\earyst
\end{claim}
\begin{proof}
We set $x = \pi(z)$. Then we have 
\aryst
\pi(a^{-1}\cdot z) = a^{-1}\cdot \pi(z) = a^{-1} \cdot x.
\earyst
 Then 
\aryst
F \Phi(z) = (a^{-1}\cdot x,  \iota_{\psi(a^{-1}\cdot x)} g_{a^{-1}\cdot \psi(x)}(\ta(a^{-1}) ( g_x(z))))
\earyst
and
\aryst
\Phi( \ta(a^{-1})( z)) =  (a^{-1}\cdot x, \iota_{\psi(a^{-1}\cdot x)} g_{a^{-1}\cdot x}(\ta(a^{-1})( z))).
\earyst
Then by definition, it suffices to show that
\aryst
a u_{a^{-1}\cdot x} = u_x a u_{a^{-1}\cdot \psi(x)}.
\earyst
By definition,
\aryst
a u_{a^{-1}\cdot x}\cdot \psi(a^{-1}\cdot x) = a \cdot a^{-1}\cdot x = x.
\earyst
We also notice that $a^{-1}\cdot \psi(x) \in a^{-1} \cdot \xi(x) \subset \xi(a^{-1}\cdot x)$. Thus 
\aryst
\psi(a^{-1}\cdot \psi(x)) = \psi(a^{-1}\cdot x).
\earyst
Then
\aryst
 u_x a u_{a^{-1}\cdot \psi(x)}\cdot \psi(a^{-1}\cdot x) = u_x a  \cdot a^{-1}\cdot \psi(x) = x.
\earyst
This completes the proof.
\end{proof}

Let $\{ \mu^*_x \}$ be the disintegration of $\mu^*$ with respect to the partition of $G / \Gamma \times M$ into fibers.
The following properties follow  immediately from the above constructions and observations.

\begin{prop}\label{prop property of F}
We have
\enmt
\item for a.e. $x \in G / \Gamma$ and every $x' \in \xi(x)$, $F_x = F_{x'}$; in particular, $x \mapsto F_x$ is $\xi$-measurable.

\item  The function $x \mapsto \log \norm{F_x^{-1}}_{C^1}$ is in $L^1(G / \Gamma, \nu)$.

\item $\Phi$ is a measurable conjugacy between the dynamics of $a^{-1}$ on $M^\a$ and of $F$ on $G/ \Gamma \times M$.

\item The fiberwise Lyapunov exponents for $Da$ with respect to $\mu$ are all non-positive if, and only if, the fiberwise Lyapunov exponents of $F$ with respect to $\mu^*$ are all non-negative.

\item $\mu$ is $H^u$-invariant if and only if the map $x \mapsto \mu^*_x$ is $\xi$-measurable.

\eenmt
\end{prop}
\begin{proof}
Item $(1)$ follows immediately from the construction. Item $(3)$ is given by Claim \ref{claim. boreliso}. 
Item $(4)$ follows from item (3) and our hypothesis on $a$ in Proposition \ref{prop. inv}: all fiberwise Lyapunov exponents of $D\tilde\alpha(a)$ are non-positive.

To show item (2), we first notice that by \eqref{def F_x} for $\nu$-a.e. $x \in G/\Gamma$, we have
\aryst
\| F_{x}^{-1} \|_{C^1} 
&\leq& \| \tilde\alpha(u_{\alpha^{-1}\cdot \psi(x)})|_{\pi^{-1}(\psi(a^{-1} \cdot x))} \|_{C^1}  \| \tilde\alpha(a)|_{\pi^{-1}(a^{-1} \cdot \psi(x))} \|_{C^1}.
\earyst
By Lemma \ref{prop integrability}, we have
\aryst
     \log \| \tilde\alpha(a)|_{\pi^{-1}(a^{-1} \cdot \psi(x))} \|_{C^1}
&\leq& C d(a, e) + C d(a^{-1} \cdot \psi(x), x_0) + C,  \\
    \log\| \tilde\alpha(u_{\alpha^{-1}\cdot \psi(x)})|_{\pi^{-1}(\psi(a^{-1} \cdot x))} \|_{C^1} 
&\leq& C \sup_{b \in W}d(b, e) + C d(  \psi(a^{-1} \cdot x), x_0) + C. 
\earyst
Note that there are $u_1, u_2 \in W$ such that 
\aryst
a^{-1} \cdot \psi(x) = a^{-1} u_1^{-1} x, \quad \psi(a^{-1} \cdot x) =u_2^{-1} a^{-1} x.
\earyst
Then we have
\aryst
d(a^{-1} \cdot \psi(x), x_0),  d(  \psi(a^{-1} \cdot x), x_0) \leq C d(x, x_0) + C'
\earyst
for some $C$ depending only on $G, \Gamma$, and some $C'$ depending only on $W$ and $a$.
We now briefly explain why we have that 
\ary \label{eq quasiboundedfd}
(x \mapsto d(x, x_0)) \in L^1(G/\Gamma, \nu).
\eary
Let  $\rho: G \to SL(N, \R)$ be an embedding given in \cite[Chapter VIII, Section 1]{Mar}.
By \cite[3.5 (*)]{LMR},   we have
\aryst
d(g \Gamma,  ag \Gamma) \leq d(g,  ag) \leq C(1 + \log \| \rho(a) \| ).
\earyst
We remind the reader that the first $d$ denotes the quotient metric on $X$, and the second $d$ denotes the right-invariant metric on $G$.
We deduce \eqref{eq quasiboundedfd} by \cite[Chapter VIII, Section 1, Proposition 1.2]{Mar}.
Then item (2) follows suit.

The \lq\lq only if" part of Item $(5)$  follows by definition. We assume that $x \mapsto \mu^*_x$ is $\xi$-measurable. Then for $\mu$-a.e. $x$, for any $h \in H^u$ such that $h(\pi(x)) \in \xi(\pi(x))$, we have $\ta(h)_*\mu_{\pi(x)} = \mu_{h(\pi(x))}$ where $\{\mu_z\}_{z \in G/\Gamma}$ is the disintegration of $\mu$ along the fibers. Moreover by Claim \ref{claim. boreliso}, we see that $x \mapsto \mu^*_x$ is $a^{n}(\xi)$-measurable for any $n \geq 1$. We can use the above argument for $a^n(\xi)$ instead of $\xi$ (for all $n\geq 1$) to show that $\mu$ is $H^u$-invariant.
\end{proof}

\subsection{Avila-Viana's invariance principle}

We will use a variant of \cite[Theorem B]{AV} to conclude the proof of Proposition \ref{prop. inv}. Let us first briefly recall the setting in \cite{AV}.

Let $(\hat X, \hat{\cB}, \hat \mu )$ be a probability space, and let $\hat f: \hat X \to \hat X$ be an invertible $\hat \mu$-preserving measurable transformation. Let $N$ be a compact Riemannian manifold. We set $\hat \cE = \hat X \times N$, and denote by
 $\hat P: \hat \cE \to \hat X$ the projection to the first coordinate. We say that a $\hat{\cB} \otimes \cB_N$-measurable transformation $\hat F: \hat \cE \to \hat \cE$ is a {\it smooth cocycle} over $\hat f$ if $\hat F$ is of form
 $\hat F(\hat x, \hat y) = (\hat f(\hat x), \hat F_{\hat x}(\hat y))$, where $\hat F_{\hat x}$ is a diffeomorphism of $N$ for each $\hat x$. We also assume the following:
 \ary \label{eq integrablederivatives}
 \int  |\log ( \sup_{\hat y}\| D \hat F_{\hat x}(\hat y)^{-1} \|)| d\hat\mu(\hat x) < \infty.
 \eary
 In the following,  
 for any integer $k$, for any $\hat x \in {\hat X}$ we define
    \aryst
    \hat  F^{k}_{\hat x} = \begin{cases}
    {\hat F}_{{\hat f}^{k-1}({\hat x})} \cdots {\hat F}_{\hat x} \quad & k \geq 0, \\
    ({\hat F}_{{\hat f}^{-k}({\hat x})})^{-1} \cdots ({\hat F}_{{\hat f}^{-1}({\hat x})})^{-1} \quad & k < 0.
    \end{cases}
    \earyst
 We warn the readers not to confuse the above notation with $(\hat F_{\hat x})^k$.

 In this case, for any $\hat F$-invariant probability measure $\hat m$ on $\hat \cE$ that projects to $\hat \mu$ under $\hat P$, the minimal Lyapunov exponent is a well-defined quantity at $\hat m$-almost every $(\hat x, \hat y)$ by the following formula:
 \aryst
 \lambda_-(\hat F, \hat x, \hat y) = \lim_{n \to \infty} \frac{1}{n} \log \| D\hat F_{\hat x}^n(\hat y)^{-1} \|^{-1}.
 \earyst

The following theorem, whose proof is deferred to Appendix \ref{aapendix}, is a variant of \cite[Theorem B]{AV}.
\begin{thm}\label{thm invppl}
	Let $\hat m$ be an $\hat F$-invariant measure on $\hat \cE$ which projects down to $\hat \mu$.
	Let $\cB_0 \subset \hat \cB$ be a $\sigma$-algebra which generates $\hat \cB \mod 0$ under $\hat f$. Assume that both $\hat f$ and $\hat x \mapsto \hat F_{\hat x}$ are $\cB_0$-measurable $\mod 0$, and $\lambda_-(\hat F, \hat x, \hat y) \geq 0$ for $\hat m$-almost every $(\hat x, \hat y)$, then the disintegration $\hat x \mapsto \hat m_{\hat x}$ of the measure $\hat m$ is $\cB_0$-measurable $\mod 0$.
\end{thm}

\subsection{Completing the proof}
We can now finish the proof of Proposition \ref{prop. inv}.

\begin{proof}[Proof of Proposition \ref{prop. inv}]

By Proposition \ref{prop property of F}, the hypothesis of Theorem \ref{thm invppl} is satisfied with $(\hat X, N, \cB_0,  \hat \mu, \hat m,  \hat f, \hat F)$ being $(G/\Gamma, M, \cB_{\xi}, \mu, \nu, a^{-1}, F)$. Here $\cB_\xi$ denotes the complete $\sigma$-algebra generated by the partition $\xi$.
Then by Theorem \ref{thm invppl}, the map $x \mapsto \mu^*_x$ is $\xi$-measurable. Proposition \ref{prop. inv} then follows from Proposition \ref{prop property of F}(5).
\end{proof}

\section{Proof of Proposition \ref{prop subexp and invariant metric}}\label{sec: Subexponential growth and invariant metric}

Recall that we fixed a finite set of symmetric generators $\{\gamma_i\}$ for $\Gamma$. The word distance $\ell$ on $\Gamma$ is defined in Section \ref{sec: Suspension constuction and its properties}.

\begin{proof}[Proof of Proposition \ref{prop subexp and invariant metric}]
We let $\norm{\cdot}_g$ denote the background Riemannian metric $g$ on $TM$, and let $\vol_g$ denote the volume form induced by $\norm{\cdot}_g$. There is a $C^\infty$ Riemannian metric on $S^{2}(T^*M)$ associated to $\norm{\cdot}_g$. 
We denote by $L^{p}(M, \vol_g, S^{2}(T^*M))$ the space of $L^{p}$ sections of the tensor bundle $S^{2}(T^*M)$ with respect to $\vol_g$.

Since $\alpha$ has  uniform subexponential growth  of derivatives, by the strong property (T) of the lattice $\Gamma$ (proved in \cite{La, dLdl, dlS}), we can adapt the argument in \cite{BFH1} to show that there exist:\footnote{We obtain $(1)$--$(3)$ for $p \in (1,\infty)$ by strong property $(T)$, then the case for $p=1$ follows from Cauchy's inequality.}
\enmt
\item constants $C'''_p, \sigma_p > 0$ for every $1 \leq p < \infty$;
\item $\overline{g} \in L^{p}(M, \vol_g, S^{2}(T^*M))$ for all $1 \leq p < \infty$, which is non-degenerate, i.e., $\norm{v}_{\overline{g}} > 0$ for $\vol_g$-a.e. $x \in M$, and every non-zero $v \in T_xM$;
\item a sequence of probability measures on $\Gamma$, denoted by $\{\omega_n\}_n$, satisfying $supp(\omega_n) \subset B_{word}(e, n) \subset \Gamma$ for every $n$, where  $B_{word}(e, n)$ denotes the radius $n$ open ball in $\Gamma$ centered at $e$ with respect to the word distance,
\eenmt
 such that, setting $g_n = \int \alpha(\gamma)^{*}g d\omega_n(\gamma)$,
then we have
\ary \label{item g n overline g}
\norm{g_n - \overline{g}}_{L^{p}} < C'''_p e^{-n\sigma_p},\quad \forall 1 \leq p < \infty.
\eary
As a consequence, denote by $\vol_{\overline{g}}$ the measurable volume form induced by $\norm{\cdot}_{\overline{g}}$, then the measure $d\vol_{\overline{g}}$ is absolutely continuous with respect to $d\vol_g$, and the density function $\frac{d\vol_{\overline{g}}}{d\vol_g}$ has full support.

We define Lebesgue measurable functions $\overline{R}, \underline{R} : M \to \R_+$ as follows. Set
\aryst
\overline{R}(x) = \sup_{v \in T_xM, \norm{v}_g = 1} \norm{v}_{\overline{g}}, \quad \underline{R}(x) = \inf_{v \in T_xM, \norm{v}_g = 1} \norm{v}_{\overline{g}}.
\earyst
It is direct to see that for $d\vol_g$-a.e. $x \in M$,
\aryst
\frac{d\vol_{\overline{g}}}{d\vol_g}(x) < \overline{R}^{\dim M}(x), \quad \frac{d\vol_g}{d\vol_{\overline{g}}}(x) < \underline{R}^{-\dim M}(x).
\earyst
We have the following lemma.
\begin{lemma} \label{lem underline R overline R integrable}
For every $1 \leq p < \infty$, there is $C_p > 0$ such that  
\aryst
\int \underline{R}^{-p} d\vol_g < C_p,  \quad \int \overline{R}^p d\vol_g < C_p.
\earyst
\end{lemma}
\begin{proof}
The second inequality follows immediately from the fact that $\overline{g} \in L^{p}(M, \vol_g, S^{2}(T^*M))$. It remains to prove the first inequality.

We define for every $n \geq 1$,
\aryst
&& \underline{R}_n(x) = \inf_{v \in T_xM, \norm{v}_{g} = 1}\norm{v}_{g_n}, \quad \forall x \in M, \\
\mbox{and }&&\Omega_n = \{x \mid \underline{R}(x) \geq \frac{1}{2}\underline{R}_n(x)\}.
\earyst
For the convenience of the notation, we set $\Omega_0 = \emptyset$. 
It is clear that $\cup_n \Omega_n$ is a $d\vol_g$-conull subset of $M$.

By the uniform subexponential  growth of  derivatives, for every $\varepsilon > 0$ there is $C''_{\varepsilon} > 0$ such that
\ary \label{item upper bound underline R n}
\sup_{x \in M}  (\underline{R}_{n}(x)^{-1}) < C''_{\varepsilon}e^{n\varepsilon}, \quad \forall n \geq 1.
\eary
By \eqref{item g n overline g} and \eqref{item upper bound underline R n}, for every $\varepsilon > 0$  we have
\aryst
\vol_g(\Omega_n^{c}) &\leq& \vol_g( \{ x \mid |\underline{R}(x) - \underline{R}_n(x)| > \frac{1}{2} \underline{R}_n(x) \} ) \\
&\leq& 2\sup_{x \in M}  (\underline{R}_{n}(x)^{-1})\int |\underline{R}(x) - \underline{R}_n(x)| d\vol_g(x) \\
&\leq& 2C''_{\varepsilon}C'''_1 e^{n\varepsilon-n\sigma_1}.
\earyst
Then for each $1 \leq p < \infty$, we take $\varepsilon = \sigma_1/(10p)$, and we obtain
\aryst
\int \underline{R}(x)^{-p} d\vol_g(x) 
&\leq& 2^{p} \sum_{n= 0}^{\infty} \int_{\Omega_{n+1} \setminus \Omega_n} \underline{R}_{n+1}(x)^{-p} d\vol_g(x) \\
&\leq& 2^{p} \sum_{n= 0}^{\infty} \sup_x( \underline{R}_{n+1}(x)^{-p}) \vol_g(\Omega_n^{c})\\
&\leq&2^{p+1} (C''_{\varepsilon})^{p+1}C'''_1 \sum_{n=0}^{\infty} e^{(n+1)p\varepsilon - n(\sigma_1 - \varepsilon)} := C_p < \infty.
\earyst
\end{proof}

\begin{lemma}
For every $1 \leq p < \infty$, there exists $D_p > 0$ such that for every $\gamma \in \Gamma$,
\aryst
\int_M \norm{D_x\a(\gamma)}_{g}^{p} d\vol_g(x) \leq  D_p.
\earyst
\end{lemma}
\begin{proof}
Take an arbitrary $\gamma \in \Gamma$, and set $F = \a(\gamma)$. 
We recall that $F$ preserves  $\overline{g}$. That is, for $d\vol_g$-a.e. $x$, for every $v \in T_xM$, we have $\norm{v}_{\overline{g}} = \norm{D_xF(v)}_{\overline{g}}$. Hence the measure $d\vol_{\overline{g}}$ is $F$-invariant.

Notice that for $d\vol_g$-a.e. $x \in M$,\aryst
&&\norm{D_xF}_{g} = \sup_{v \in T_xM, \norm{v}_g = 1}\norm{D_xF(v)}_{g} \\
&=& \sup_{v \in T_xM, \norm{v}_g = 1}\norm{D_xF(v)}_{\overline{g}}\frac{\norm{D_xF( v)}_{g} }{\norm{D_xF(v)}_{\overline{g}} }  \\
&=& \sup_{v \in T_xM, \norm{v}_g = 1}\norm{ v}_{\overline{g}}\frac{\norm{D_xF( v)}_{g} }{\norm{D_xF(v)}_{\overline{g}} } \\
&\leq& \overline{R}(x) \underline{R}(F(x))^{-1}.
\earyst
Then by Cauchy's inequality,
\aryst
\int \norm{D_xF}_g^{p} d\vol_g(x) \leq \left(\int \overline{R}(x)^{2p} d\vol_g(x) \right)^{1/2} \left(\int \underline{R}(F(x))^{-2p} d\vol_g(x) \right)^{1/2}.
\earyst
Also
\aryst
\int \underline{R}(F(x))^{-2p} d\vol_g(x) &=& \int \underline{R}(F(x))^{-2p}  \frac{d\vol_g}{d\vol_{\overline{g}}}(x) d\vol_{\overline{g}}(x)  \\
&\leq& \left( \int \underline{R}(F(x))^{-4p} d\vol_{\overline{g}}(x)  \right)^{1/2} \left( \int (\frac{d\vol_g}{d\vol_{\overline{g}}}(x) )^2 d\vol_{\overline{g}}(x)  \right)^{1/2}  \\
&\leq&\left( \int \underline{R}(F(x))^{-4p} d\vol_{\overline{g}}(x)  \right)^{1/2} \left( \int \frac{d\vol_g}{d\vol_{\overline{g}}}(x) d\vol_{g}(x)  \right)^{1/2} \\
&\leq&\left( \int \underline{R}(x)^{-4p} d\vol_{\overline{g}}(x)  \right)^{1/2} \left( \int  \underline{R}^{-\dim M}(x) d\vol_{g}(x)  \right)^{1/2}
\earyst
and
\aryst
 \int \underline{R}(x)^{-4p} d\vol_{\overline{g}}(x)  &=&  \int \underline{R}(x)^{-4p} \frac{d\vol_{\overline{g}}}{d\vol_g}(x)  d\vol_{g}(x)  \\
 &\leq& \left( \int \underline{R}(x)^{-8p} d\vol_{g}(x)  \right)^{1/2} \left( \int (\frac{d\vol_{\overline{g}}}{d\vol_g}(x) )^2 d\vol_{g}(x)  \right)^{1/2} \\
  &\leq& \left( \int \underline{R}(x)^{-8p} d\vol_{g}(x)  \right)^{1/2} \left( \int \overline{R}^{2\dim M}(x)  d\vol_{g}(x)  \right)^{1/2}.
\earyst
By Lemma \ref{lem underline R overline R integrable}, 
\aryst
\int \norm{D_xF(v)}_g^{p} d\vol_g(x) \leq C_{2p}^{1/2} C_{8p}^{1/8} C_{2\dim M}^{1/8} C_{\dim M}^{1/4}.
\earyst
Since $\gamma$ is chosen arbitrarily, we can conclude the proof by taking $D_p$ to be the right hand side of the last  inequality.
\end{proof}
We fix an embedding $\iota: M  \rightarrow \R^{N}$ for some integer $N$. Let $\pi_i: \R^N \to \R$ be the projection to the $i$-th coordinate. We have seen that for every $1 \leq p < \infty$, there exists a constant $C'_p > 0$ such that for every $1 \leq i \leq N$, for every $\gamma \in \Gamma$, 
\aryst
\int |D_x(\pi_i \iota \a(\gamma))|^{p} d\vol_g(x) < C'_p.
\earyst
Take $p > \dim M/(1-\beta)$. Then by Sobolev's embedding theorem, we can see that the set $\{\alpha(\gamma) \mid \gamma \in \Gamma \}$ is pre-compact in ${\rm Hol}$-${\rm Homeo}^{\beta}(M)$. We know that any pre-compact subset of ${\rm Hol}$-${\rm Homeo}^{\beta}(M)$ is  equicontinuous in ${\rm Homeo}(M)$. 
Thus the closure of $\a(\Gamma)$ in ${\rm Homeo}(M)$ is a compact topological group $K_0$, and it is direct to verify  by definition that $K_0 \subset  {\rm Hol}$-${\rm Homeo}^{\beta}(M)$. It is then direct to construct a $\Gamma$-invariant continuous distance on $M$ by averaging.
\end{proof}

\appendix
\section{} \label{aapendix}
We now give the proof of Theorem \ref{thm invppl} in this appendix.
We recall the construction in \cite[Section 3]{AV}.
There is a Lebesgue space $(X, \cB, \mu)$ obtained by identifying any two points of $\hat X$ which are not distinguished by any element of $\cB$; and a projection $\pi : \hat X \to X$ such that $\cB = \pi_* \cB_0$ and $\mu = \pi_* \hat \mu$. Since $\hat f$ is $\cB_0$-measurable $\mod 0$, there exists a $\cB$-measurable $\mod 0$ transformation $f: X \to X$ such that $\pi \circ \hat f = f \circ \pi$. Let $\cE = X \times N$ and $P : \cE \to X$ the canonical projection. Since $\hat F$ is $\cB_0$-measurable $\mod 0$, we may write $F_{\pi(\hat x)} = \hat F_{\hat x}$ for some $\cB$-measurable $\mod 0$ fiber bundle morphism $F: \cE \to \cE$ over $f$. The measure $m = (\pi \times id)_* \hat m$ is $F$-invariant and projects down to $\mu$. Denote by $\{ m_x \}_{x \in X}$  the measure disintegration of $m$  corresponding to the partition of $\cE$ into the fibers. By the $F$-invariance of $m$ we deduce that for $\mu$-a.e. $x \in X$,
    \ary \label{eq mfxm}
    m_{f(x)} = \int (F_{x'})_* m_{x'} d\mu^{f^{-1}(\cB_0)}_x(x'). 
    \eary 

For any integer $l  \geq 0$, we define $J_l : \cE \to [0, \infty)$ by considering the Lebesgue decomposition of $(F_x^{-l})_* m_{f^l(x)}$ relative to $m_x$ :
\aryst
(F_x^{-l})_* m_{f^l(x)} = J_l(x, \cdot) m_x + \eta^{(l)}_x.
\earyst
We abbreviate $J_1$ as $J$.

Define
\aryst
h(F, m) = \int -\log J dm.
\earyst
Following the proof of  \cite[Theorem B]{AV},
we will show that $m(\{ J = 0\}) = 0$
and in addition the following is true.
\begin{prop}\label{prop ineq_entropyandle}
	We have
	\aryst
	0 \leq  h(F, m) \leq - \dim N \int \min\{ 0, \lambda_-(\hat F) \}  d\hat m.
	\earyst
\end{prop}

The statement of Proposition \ref{prop ineq_entropyandle} is the same as \cite[Proposition 3.1]{AV}, except that we are now assuming \eqref{eq integrablederivatives} while in \cite{AV} the authors assume that $\log \| D \hat F_{\hat x}(\hat y)^{-1} \|$,  $\log \| D \hat H_{\hat x}(\hat y) \|$ and $\log \| D \hat H_{\hat x}(\hat y)^{-1} \|$ are all uniformly bounded, and the dependence of $D\hat F_{\hat x}(\hat y)$, $D\hat H_{\hat x}(\hat y)$ on $\hat x, \hat y$ are uniformly continuous. Thus we will need to make some adjustments to the proof in \cite{AV} (see also \cite{Led2}).

Under the hypothesis of Theorem \ref{thm invppl}, we can conclude by Proposition \ref{prop ineq_entropyandle} that $h(F, m)$ vanishes.
Once we know that $h( F,  m)$ vanishes, we can apply \cite[Proposition 3.2]{AV} to conclude the proof of Theorem \ref{thm invppl}. We recall the statement below.
\begin{prop}\label{prop zeroentropy_deterministic}
If $h(F,  m) = 0$ then $\hat x \mapsto \hat m_{\hat x}$ is $\cB_0$-measurable $\mod 0$.
\end{prop}
The proof of Proposition \ref{prop zeroentropy_deterministic} is rather general and the condition \eqref{eq integrablederivatives} suffices. \clb Now it suffices to give the proof of 
Proposition \ref{prop ineq_entropyandle}. The proof here follows essentially the scheme in \cite{Led2}.
\begin{proof}[Proof of Proposition \ref{prop ineq_entropyandle}]
	By the same argument in \cite[Section 3.2]{AV}, we may assume without loss of generality that $\hat m$ is ergodic for $\hat F$. In this case, $\min(0, \lambda_-(\hat F))$ is a constant $\hat m$-almost everywhere, and is denoted by $-\lambda \leq 0$.


    For any integer $k$, for any $(x, \xi) \in X \times N$ we define
    \aryst
    F^{k}_x = \begin{cases}
    F_{f^{k-1}(x)} \cdots F_x \quad & k \geq 0, \\
    F_{f^{-k}(x)}^{-1} \cdots F_{f^{-1}(x)}^{-1} \quad & k < 0,
    \end{cases}
    \earyst
    and
    \aryst
    L_k(x, \xi) = \| D_\xi F^{-k}_{f^k(x)} \|, \quad
     C_k(x) = \sup_{\xi \in N} L_k(x, \xi),  \quad \tilde C_k(x, \xi) = C_k(x). 
    \earyst
    Notice that we have
    \ary \label{eq ckc1}
    0 \leq \log \tilde C_k(x, \xi) \leq \sum_{i=0}^{k-1} \log \tilde C_1(F^{i}(x,\xi)).
    \eary
    
    Given $(x, \xi) \in X \times N$, we denote by $B(\xi, \delta)$ the ball in $N$ centered at $\xi$ of radius $\delta > 0$ and write
    \aryst
    B((x, \xi), \delta) = \{ x \} \times B(\xi, \delta).
    \earyst
    For each integer $l \geq 0$, we write 
    \aryst
    J_l(x, \xi; \delta) = \frac{(F_{f^l(x)}^{-l})_* m_{f^l(x)}(B(\xi, \delta))}{m_x(B(\xi, \delta))}
    \earyst
    and
    \aryst
    J^*_l(x, \xi) = \max_{\delta > 0} J_l(x, \xi; \delta).
    \earyst
    It is clear that $J_l \geq 0$ and $J^*_l \geq 1$.  
    
    We fix some $\epsilon > 0$. Then there is $\beta_1 = \beta_1(\epsilon) > 0$ so that for any set $A \subset X \times N$ with $m(A) < \beta_1$, we have
    \ary  \label{eq intc1ona}
    \int_A \log \tilde C_1 d\mu < \epsilon.
    \eary
   Fix some integer $l > 0$ such that the measurable set $\Lambda_1 \subset X \times N$ defined by
    \aryst
    \Lambda_1 = \{ (x, \xi)  \mid L_l(x, \xi) \leq e^{(\lambda + \epsilon)l}  \}
    \earyst
    satisfies $m(\Lambda_1) > 1 - \beta_1/2$. 
    Then there is a subset $\Lambda \subset \Lambda_1$ with 
    \ary \label{eq mlambda}
    m(\Lambda) > 1 - \beta_1
    \eary
    such that the derivatives $D_{\xi}F_x$ are uniformly continuous in $\xi$ over all $x \in \Lambda$, and for some $\delta_1 = \delta_1(\epsilon, l, \Lambda) > 0$, for any $x \in \Lambda$, for any $\delta \in (0, \delta_1(\epsilon))$ we have
    \ary \label{eq flmapsinclusion}
F_{x}^{-l}(B(\xi, \delta)) \subset B(F_{x}^{-l}(\xi), e^{(\lambda + 2\epsilon)l}\delta).
    \eary
    We denote by $E_l$ the collection of ergodic component of $m$ for $F^l$. Since $\mu$ is $F$-ergodic, we deduce that $E_l$ is finite and $F$ induces a cyclic permutation of $E_l$. Moreover, for $m$-almost every $(x,\xi)$ we denote by $m_{(x,\xi)}$ the ergodic component at $(x,\xi)$.
    
    By \cite[Prop 5]{Led2}, we know that 
    \aryst
    \log J^*_l \in L^1(\cE, m).
    \earyst
    More precisely, we have the following.
    \begin{lemma}\label{lem jl} 
    For $m$-a.e. $(x, \xi)$, we have
    \aryst
    J_l(x, \xi) = \prod_{i=0}^{l-1} J(F^i(x, \xi)).
    \earyst
    Consequently, for any $m' \in E_l$, we have
    \ary  \label{eq lhfm}
    l h(F, m) = - \int \log J_l(x, \xi) dm'(x, \xi).
    \eary 
    \end{lemma}
    \begin{proof}
    It is clear that the first equality holds when $l=1$. For any $l > 1$, we have
    \aryst
    (F^{-l}_{f^l(x)})_* m_{f^l(x)} &=& (F^{-1}_x)_*(F^{-l+1}_{f^l(x)})_* m_{f^l(x)} \\
    &=& (F^{-1}_x)_*(J_{l-1}(F(x, \xi))m_{f(x)} + \eta^{(l-1)}_{f(x)}) \\
    &=& J_{l-1}(F(x, \xi)) \cdot J(x,\xi) m_x + J_{l-1}(F(x, \xi)) \eta_x + (F^{-1}_x)_*\eta^{(l-1)}_{f(x)}.
    \earyst
    Here $\eta^{(l-1)}_{f(x)}$ is the singular component of $(F^{-l+1}_{f^l(x)})_* m_{f^l(x)}$ with respect to $m_{f(x)}$.
    By definition, $\eta_x$ is singular with respect to $m_x$. Moreover, $ (F^{-1}_x)_*\eta^{(l-1)}_{f(x)}$ is also singular with respect to $m_{x}$ for $m$-a.e $x$.
    Otherwise, we would know that $\eta^{(l-1)}_{f(x)}$ is not singular with respect to $(F_x)_*m_{x}$; then by \eqref{eq mfxm} we would know that $\eta^{(l-1)}_{f(x)}$ is not singular with respect to $m_{f(x)}$. A contradiction.
    Consequently, we see that
    \aryst
    J_l = J_{l-1} \circ F \cdot J.
    \earyst
    We then conclude the proof of the first equality by induction.
    The equality \eqref{eq lhfm} in the lemma is an immediate consequence of the first equality, and the fact that $m = \frac{1}{l}\sum_{i=0}^{l-1} (F^i)_*m'$ for any $m' \in E_l$.
     \end{proof}
    We choose some $\beta_2 = \beta_2(\epsilon, l)  > 0$ such that for  $m' \in E_l$, and for every $A \subset X \times N$ with $m'(A) < \beta_2$, we have
    \ary \label{eq inajstar}
    \int_A (\log J^*_l + \log J_l) dm' < \epsilon.
    \eary    
   
   We define
   \aryst
   Z = \{  (x, \xi) \mid J_l(x, \xi) = 0 \}, \quad G = Z^c =  \{ (x, \xi) \mid J_l(x,\xi) > 0 \}.
   \earyst
   We fix a large constant $D > 0$.
   Given a constant $\delta  > 0$, we define
   \aryst
   G(\delta) &=& \{ (x, \xi) \in G \mid \log J_l(x, \xi; \delta') \leq \log J_l(x, \xi) + \epsilon \quad \forall \delta' \in (0, \delta) \}, \\
   Z(\delta) &=& \{ (x, \xi) \mid \log J_l(x, \xi; \delta') \leq -D \quad \forall \delta' \in (0, \delta) \}.
   \earyst
   We fix some $\delta_2 = \delta_2(\epsilon, l) > 0$ sufficiently small so that for every $m' \in E_l$,
   \ary \label{eq mgdelta2}
   m'(G \setminus G(\delta_2)) < \beta_2.
   \eary
   
   We take an arbitrary $\delta_0 \in (0, \min(\delta_1, \delta_2))$.
   Given $(x, \xi) \in X \times N$, define $\delta_l(x, \xi; 0) = \delta_0$ and for $k \geq 1$ recursively define
   \aryst
   \delta_l(x, \xi; k) = \begin{cases} 
   e^{(-\lambda - 2\epsilon)l} \delta_l(x, \xi; k-1) \quad & F^{kl}(x, \xi) \in \Lambda, \\
   [\tilde C_l(F^{kl}(x, \xi))]^{-1} \delta_l(x, \xi; k-1) \quad & F^{kl}(x, \xi) \notin \Lambda.
   \end{cases}
   \earyst
   Observe that we have
   \aryst
   \delta_l(x, \xi; k+1) \leq \delta_l(x, \xi; k) \leq \delta_0 \quad \forall k \geq 0
   \earyst
   and by \eqref{eq flmapsinclusion} and the definition of $\tilde C$ we deduce that
   \aryst
   F^{-l}(B(F^{(k+1)l}(x, \xi), \delta_l(x, \xi; k+1))) \subset B(F^{kl}(x, \xi), \delta_l(x, \xi; k)) \quad k \geq 0.
   \earyst
   \begin{lemma}
   For $m$-a.e. $(x, \xi)$, we have
   \aryst
   \liminf_{n \to \infty} \frac{1}{n} \log \delta_l(x, \xi ; n) \geq (- \lambda - 3 \epsilon) l.
   \earyst
   \end{lemma}
   \begin{proof}
   By definition, we have
   \aryst
   \log \delta_l(x, \xi ; n) = \log \delta_0 + \sum_{k=0}^{n-1} (l(-\lambda - 2\epsilon)1_{F^{kl}(x,\xi) \in \Lambda} - \log \tilde C_l(F^{kl}(x,\xi))1_{F^{kl}(x,\xi) \notin \Lambda}).
   \earyst
   Then by Pointwise Ergodic Theorem, we have
   \aryst
   \liminf_{n \to \infty} \frac{1}{n}   \log \delta_l(x, \xi ; n) &\geq&  - (\lambda + 2\epsilon)l m_{(x,\xi)}(\Lambda) - \int_{\Lambda^c} \log \tilde C_l dm_{(x, \xi)}.
   \earyst
   By \eqref{eq ckc1}, \eqref{eq intc1ona} and \eqref{eq mlambda}, we obtain
   \aryst
   \int_{\Lambda^c} \log \tilde C_l dm_{(x, \xi)} &\leq&
   \sum_{i=0}^{l-1} \int_{\Lambda^c} \log \tilde C_1 \circ F^i dm_{(x, \xi)} \\
   &\leq&  \sum_{i=0}^{l-1} \int_{\Lambda^c} \log \tilde C_1   d(F^i)_*m_{(x, \xi)} \\
   &=& l \int_{\Lambda^c} \log \tilde C_1   dm \leq l \epsilon.
   \earyst
   The equality follows from $m = \frac{1}{l} \sum_{i=0}^{l-1}(F^i)_*m_{(x, \xi)}$ for $m$-a.e. $(x, \xi)$.
   Then
   \aryst
    \liminf_{n \to \infty} \frac{1}{n} \log \delta_l(x, \xi ; n)  \geq  - (\lambda + 3\epsilon) l.
   \earyst
   \end{proof}
   Recall that by \cite[Proposition 5]{Led2}, for any Borel probability measure $\nu$ on $N$, we have
   \aryst
   \limsup_{r \to 0} \frac{\log \nu(B(\xi, r))}{\log r} \leq \dim N, \quad \nu-a.e. \ \xi.
   \earyst
   Thus we may pick a subset $\Omega_1 \subset X \times N$ such that 
\aryst
\limsup_{r \to 0} \frac{ \log \inf_{(x, \xi) \in \Omega_1} m_{x}(B(\xi, r))}{\log r} \leq \dim M  
\earyst
and $m'(\Omega_1) > 0$ for every $m' \in E_l$.
Then for $m$-almost every $(x, \xi)$, there is an infinite sequence of $n$ such that $F^{nl}(x, \xi) \in \Omega_1$. Then for all sufficiently large $n$ in such sequence we have
\ary
\frac{1}{n}\log m_{f^{nl}(x)}(B(F^{nl}(x, \xi), \delta_l(x, \xi; n))) 
&\geq& \frac{1}{n} \log \delta_l(x, \xi; n) (\dim N + \epsilon)  \nonumber \\
&\geq& ( - \lambda - 4\epsilon)(\dim N + \epsilon)l. \label{eq lowerboundforlogm}
\eary
On the other hand, we have
\aryst
&& m_{f^{nl}(x)}(B(F^{nl}(x, \xi), \delta_l(x, \xi; n))) \\
&=& m_x(B(x, \delta_l(x, \xi; 0))) \\
&& \cdot \prod_{j=0}^{n-1} \frac{m_{f^{(j+1)l}(x)}(B(F^{(j+1)l}(x, \xi), \delta_l(x, \xi; j+1) ))}{m_{f^{jl}(x)}(B(F^{jl}(x, \xi), \delta_l(x, \xi; j)) )}  \\
&\leq&  \prod_{j=0}^{n-1} \frac{m_{f^{(j+1)l}(x)}(B(F^{(j+1)l}(x, \xi), \delta_l(x, \xi; j+1) ))}{m_{f^{jl}(x)}(B(F^{jl}(x, \xi), \delta_l(x, \xi; j)))} \\
&\leq&  \prod_{j=0}^{n-1} \frac{m_{f^{(j+1)l}(x)}(F^l(B(F^{jl}(x, \xi), \delta_l(x, \xi; j) )))}{m_{f^{jl}(x)}(B(F^{jl}(x, \xi), \delta_l(x, \xi; j)))} \\
&=& \prod_{j=0}^{n-1} J_l(F^{jl}(x, \xi), \delta_l(x, \xi; j)).
\earyst
Take an arbitrary $m'  \in E_l$. Then for $m'$-almost every $(x, \xi)$ we have
\aryst
&& \limsup \frac{1}{n} \log m_{f^{nl}(x)}(B(F^{nl}(x, \xi), \delta_l(x, \xi; n))) \\
&\leq& \limsup \frac{1}{n} \sum_{j=0}^{n-1} \log J_{l}(F^{jl}(x, \xi), \delta_l(x,\xi; j)) \\
&\leq& \int_{Z(\delta_0)} \log J_l  dm' + \int \log J^*_l dm' \\
&\leq&   - D m'(Z(\delta_0)) + l\int \log J^* dm 
\earyst
The last inequality follows from Lemma \ref{lem jl} and the definition of $Z(\delta_0)$.
Combine the above inequality with \eqref{eq lowerboundforlogm} we conclude that
\aryst
 m'(Z(\delta_0)) \leq \frac{l(\int \log J^* dm + (\lambda + 4\epsilon)(\dim N + \epsilon))}{D}.
\earyst
Since the above holds for any $\delta_0$ sufficiently small and for any  $m' \in E_l$, we deduce that 
\aryst
m(Z) \leq \limsup_{\delta_0 \to 0} m(Z(\delta_0)) \leq \frac{l(\int \log J^* dm + (\lambda + 4\epsilon)(\dim N + \epsilon))}{D}.
\earyst
By letting $D$ tend to infinity, we conclude that $m(Z) = 0$, and consequently $m(G)  = 1$.

Now notice that for $m'$-almost every $(x, \xi)$ we have
\aryst
&& \limsup \frac{1}{n} \log m_{f^{nl}(x)}(B(F^{nl}(x, \xi), \delta_l(x, \xi; n))) \\
&\leq& \limsup \frac{1}{n} \sum_{j=0}^{n-1} \log J_{l}(F^{jl}(x, \xi), \delta_l(x,\xi; j)) \\
&\leq& \int_{G(\delta_0)} ( \log J_l + l \epsilon )  dm' + \int_{G(\delta_0)^c} \log J^*_l dm' \\
\clr &\leq& \int \log J_l dm' + 5l \epsilon.
\earyst
The last inequality follows from \eqref{eq mgdelta2}, $G(\delta_2) \subset G(\delta_0)$ and \eqref{eq inajstar}.
Combine the above inequality with \eqref{eq lowerboundforlogm} and \eqref{eq lhfm} in Lemma \ref{lem jl}, we obtain
\aryst
h(F, m) - 5 \epsilon \leq (\lambda + 4\epsilon)(\dim N + \epsilon).
\earyst
Since $\epsilon$ is arbitrary, we conclude the proof of Proposition \ref{prop ineq_entropyandle}.
\end{proof}


\begin{thebibliography}{aaaa}



\bibitem{AV}  A. Avila, M. Viana, 
\newblock {Extremal Lyapunov exponents: an invariance principle and
applications,} 
\newblock {Invent. Math.}  {\bf181}(1):115–189, 2010.










\bibitem{BFH1} A. Brown, D. Fisher, S. Hurtado,
\newblock {Zimmer's conjecture: Subexponential growth, measure rigidity, and strong property (T),} 
\newblock {to appear in Annals of Mathematics.}

\bibitem{BFH2} A. Brown, D. Fisher, S. Hurtado,
\newblock {Zimmer's conjecture for actions of $\mathrm {SL}(m,\mathbb {Z})$,} 
\newblock {Inventiones mathematicae volume 221, pages 1001–1060 (2020)}.



\bibitem{BFH3} A. Brown, D. Fisher, S. Hurtado,
\newblock {Zimmer's conjecture for non-uniform lattices and escape of mass,} 
\newblock {arXiv: 2105.14541.}



\bibitem{BRHW1} A. Brown, F. Rodriguez Hertz, Z. Wang,
\newblock {Smooth ergodic theory of $\mathbb {Z}^ d $-actions,} 
\newblock {arXiv: 1610.09997.}




\bibitem{BRHW2} A. Brown, F. Rodriguez Hertz, Z. Wang,
\newblock {Invariant measures and measurable projective factors for actions of higher-rank lattices on manifolds,} 
\newblock {to appear in Annals of Mathematics.}




\bibitem{BM}M. Burger, N. Monod,
\newblock {Continuous bounded cohomology and applications to rigidity theory,} 
\newblock {Geom. Funct. Anal.}  {\bf 12} (2002), 219-280.


\bibitem{Can} S. Cantat,
\newblock {
Progr\`es r\'ecents concernant le programme de Zimmer,
[d'apr\`es A. Brown, D. Fisher, et S. Hurtado],} 
\newblock {S\'eminaire Bourbaki,} 70\`eme ann\'ee, 2017-2018, n. 1136.




 

\bibitem{CFS}  I. P. Cornfield,  S. V. Fomin,  and Ya.G. Sinai,
\newblock {Ergodic Theory,} 
\newblock {Grundlehren der mathematischen Wissenschaften (A Series of Comprehensive Studies in Mathematics)}, vol 245. Springer, New York, NY.



\bibitem{dLdl} T. de Laat,  M. de la Salle,
\newblock {Strong property (T) for higher-rank simple Lie groups,} 
\newblock {Proc. Lond.
Math. Soc.} (3) {\bf 111}(4):936-966, 2015.


\bibitem{dlS} M. de la Salle,
\newblock {Strong property (T) for higher rank lattices,} 
\newblock {Acta Math. 223(1): 151-193 (September 2019)}.




\bibitem{EK} M. Einsiedler, A. Katok,
\newblock {Rigidity of measures: the high entropy case and non-commuting foliations,} 
\newblock {Israel J. Math.,}  {\bf 148}:169-238, 2005. Probability in mathematics.

\bibitem{FS} B. Farb, P. Shalen,
\newblock {Real-analytic actions of lattices,} 
\newblock {Invent. Math.,} {\bf 135}(2), 1999:273-296.



\bibitem{Fis1} D. Fisher,
\newblock {Recent progress in the Zimmer program,} 
\newblock {arXiv.}


\bibitem{Fis2} D. Fisher,
\newblock{Groups acting on manifolds: around the Zimmer program,} 
\newblock {In Geometry, rigidity, and group actions}, Chicago Lectures in Math., pages 72-157. Univ. Chicago Press, Chicago, IL, 2011.



\bibitem{FM} D. Fisher, G. A. Margulis,
\newblock {Local rigidity for cocycles,} 
\newblock {In Surveys in differential geometry}, Vol.
VIII (Boston, MA, 2002), volume 8 of Surv. Differ. Geom., pages 191-234. Int. Press, Somerville, MA,
2003.

\bibitem{FH} J. Franks, M. Handel,
\newblock {Distortion elements in group actions on surfaces,} 
\newblock {Duke Math. J.}, {\bf 131}
(2006), no. 3, 441-468.


 
\bibitem{Ghys} \'E. Ghys,
\newblock {Actions de r\'eseaux sur le cercle,} 
\newblock { Invent. Math. }  {\bf 137}  (1999), 199-231.



\bibitem{Hu} H. Hu,
\newblock {Some ergodic properties of commuting diffeomorphisms,} 
\newblock {Ergodic Theory and Dynamical Systems,} 1993 {\bf 13}(1):73-100.


\bibitem{KRH2} A. Katok,  F. Rodriguez Hertz,
\newblock {Measure and cocycle rigidity for certain nonuniformly hyperbolic
actions of higher-rank abelian groups,} 
\newblock {J. Mod. Dyn.,} 2010 {\bf 4}(3) 487-515.



\bibitem{Knap} A. W. Knapp,
\newblock { Lie groups beyond an introduction,} 
\newblock {Progress in Mathematics,} volume 140, Birkhauser
Boston, Inc., Boston, MA, second edition, 2002.


\bibitem{La} V. Lafforgue,
\newblock {Un renforcement de la propri\'et\'e (T),} 
\newblock {  Duke Math. J.},  {\bf 143}(3): 559-602, 2008.


\bibitem{Led} F. Ledrappier,
\newblock {Proprieites ergodiques des mesures de Sinai,} 
\newblock {Publ. Math. I.H.E.S.,} {\bf 59} (1984),
163-188.





\bibitem{Led2} F. Ledrappier,
\newblock {Positivity of the exponent for stationary sequences of matrices,} 
\newblock {Lyapunov exponents (Bremen, 1984), 1986, pp. 56–73. MR850070 (87m:60160).}



\bibitem{LedStr} F. Ledrappier, J-M. Strelcyn,
\newblock {A proof of the estimation from below in Pesin's entropy formula,} 
\newblock {Ergodic Theory and Dynamical Systems,} {\bf 2} (1982),
203-219.

\bibitem{LedYou1} F. Ledrappier, L.-S. Young,
\newblock {The metric entropy of diffeomorphisms. I. Characterization of measures
satisfying Pesin's entropy formula,} 
\newblock {Ann. of Math.} (2) 1985, {\bf122}(3):509-539.

\bibitem{LedYou2} F. Ledrappier, L.-S. Young,
\newblock {The metric entropy of diffeomorphisms. II. Relations between entropy,
exponents and dimension,} 
\newblock {Ann. of Math.} (2) 1985, {\bf122}(3):540-574.



\bibitem{LiWi} L. Lifschitz, D. Witte Morris,
\newblock {Bounded generation and lattices that
cannot act on the line,} 
 Pure Appl. Math. Q. 4 (2008), no. 1, part 2, 99-126



\bibitem{LMR}  A. Lubotzky, S. Mozes, M. S. Raghunathan, 
\newblock {The word and Riemannian
metrics on lattices of semisimple groups,} 
\newblock { Inst. Hautes \'Etudes Sci. Publ. Math.}  {\bf(91)}:5–53 (2001), 2000.

\bibitem{Mal} \u{I}. Maleshich,
\newblock {The Hilbert-Smith conjecture for H\"older actions,} 
\newblock {Uspekhi Mat. Nauk }  {\bf 52}  (1997), no. 2(314), 173–174. MR 1480156 (99d:57026)



\bibitem{M} R. M\~an\'e,
\newblock {A proof of Pesin's formula,} 
\newblock {Ergodic Theory and Dynamical Systems, }1981 {\bf1}(1):95-102.


\bibitem{Mar} G. A. Margulis,
\newblock {Discrete subgroups of semisimple Lie groups,} 
\newblock {volume 17 of Ergebnisse der Mathematik
und ihrer Grenzgebiete (3) [Results in Mathematics and Related Areas (3)].} Springer-Verlag, Berlin,
1991.




\bibitem{MT} G. A. Margulis, G. M. Tomanov.
\newblock { Invariant measures for actions of unipotent groups over
local fields on homogeneous spaces,} 
\newblock { Invent. Math.}  {\bf 116} (1-3):347–392, 1994.





\bibitem{Moo}   C. C. Moore,
\newblock {Ergodicity of flows on homogeneous spaces,} 
\newblock {Amer. J. Math.}  {\bf 88} (1966), 154-178.




\bibitem{Mostow}  G. D. Mostow,
\newblock {Strong rigidity of locally symmetric spaces,} 
Princeton University Press, Princeton,
N.J., 1973. Annals of Mathematics Studies, No. 78.



\bibitem{Ne}  M.H.A. Newmann,
\newblock {A theorem on periodic transformations of spaces,} 
\newblock {Quart. J. Math.}  {\bf 2}  (1931), 1-8. 


\bibitem{Pol} L. Polterovich,
\newblock {Growth of maps, distortion in groups and symplectic geometry,} 
\newblock {Invent. Math.}, {\bf 150}
(2002), no. 3, 655-686.


\bibitem{RS}  D. Repov\u s, E.V. \u S\u cepin,
\newblock {A proof of the Hilbert-Smith conjecture for actions by Lipschitz maps,} 
\newblock {Math. Ann.}  {\bf 308} (1997), 361-364.

\bibitem{Sel}  A. Selberg,
\newblock {On discontinuous groups in higher-dimensional symmetric spaces,} 
In Contributions to
function theory (internat. Colloq. Function Theory, Bombay, 1960), pages 147-164. Tata Institute of
Fundamental Research, Bombay, 1960.

\bibitem{St} G. Stuck, \textsl{Low dimensional actions of semisimple groups}, Israel J. Math. 76 (1991), no. 1-2, 27-71.


\bibitem{Sm}  P.A. Smith,
\newblock {Transformations of finite period, III: Newman's theorem,} 
\newblock {Ann. Math.} (2)  {\bf 42}   (1941), 446-458.


\bibitem{Weil}    A. Weil,
\newblock {On discrete subgroups of Lie groups. II,} 
\newblock {Ann. of Math.}  {\bf 75} (2)  1962, 578-602.


\bibitem{Wit}  D. Witte Morris,
\newblock {Arithmetic groups of higher Q-rank cannot act on 1-manifolds,} 
\newblock {Proc. Amer. Math. Soc.}  {\bf 122} (1994), 333-340.







\bibitem{Ye}S. Ye,
\newblock {Euler characteristics and actions of automorphism groups of free groups,} 
\newblock { Algebr. Geom. Topol.
}  
Volume {\bf 18}, Number 2 (2018), 1195-1204.

\bibitem{Zim}    R. J. Zimmer,
\newblock {Strong rigidity for ergodic actions of semisimple Lie groups,} 
\newblock { Ann. of Math. }  (2), {\bf 112}(3) (1980), 511-529.


\bibitem{Zim2}    R. J. Zimmer,
\newblock { Ergodic Theory and Semisimple Groups. } 
\newblock { Birkh\"auser,}  Basel,  1984. ISBN 3-7643-3184-4, MR 0776417 (86j:22014)


 
\bibitem{ZimmerWitteMorris}    R. J. Zimmer, D. Witte Morris,
\newblock {Ergodic Theory, Groups, and Geometry,} 
 NSF-CBMS Regional Research Conferences in the Mathematical Sciences, June 22-26, 1998, University of Minnesota.
 (Regional conference series in mathematics, no. 109),
American Mathematical Society, 2008.

 

 



  
\end{thebibliography}
\end{document}